\documentclass[11pt]{article}

\usepackage{amsmath,amssymb,amsthm,mathrsfs}
\usepackage{geometry}
\usepackage{hyperref}
\usepackage{mathtools}
\mathtoolsset{showonlyrefs}
\usepackage{xcolor}

\numberwithin{equation}{section}

\newtheorem{theorem}{Theorem}[section]
\newtheorem{proposition}[theorem]{Proposition}
\newtheorem{lemma}[theorem]{Lemma}

\theoremstyle{definition}

\newtheorem{problem}[theorem]{Problem}

\newtheorem*{rubinconjecture}{Conjecture}
\newtheorem{remark}[theorem]{Remark}

\newcommand{\Sph}{\mathbb{S}}
\newcommand{\R}{\mathbb{R}}
\newcommand{\C}{\mathbb{C}}
\newcommand{\N}{\mathbb{N}}
\newcommand{\Z}{\mathbb{Z}}
\newcommand{\Q}{\mathbb{Q}}
\newcommand{\T}{\mathbb{T}}
\newcommand{\Lap}{\Delta}

\newcommand{\scr}[1]{\mathscr{#1}}

\newcommand{\nnorm}[1]{\left\| #1 \right\|}

\newcommand{\Mt}[1]{M_t^{#1}}
\newcommand{\mt}[1]{m_t^{#1}}

\title{On the small denominator problem for generalized Minkowski--Funk transforms}
\author{Rui Han and Yaghoub Rahimi}
\date{}

\begin{document}

\maketitle

\begin{center}
\textit{Dedicated to Boris Rubin on the occasion of his 80th birthday.}
\end{center}

\begin{abstract}
Rubin's generalized Minkowski--Funk transforms $M_t^\alpha$ on the sphere $\Sph^n$ give rise, for irrational radii $t=\cos(\beta\pi)$, to a small denominator problem governed by the asymptotic behavior of their spectral multipliers. We show that for Lebesgue-almost every $\beta$ the corresponding two-sine small divisor inequality has infinitely many solutions, and deduce that $(M_t^\alpha)^{-1}$ is not bounded from $\tilde H^{s+\rho+1}(\Sph^n)$ to $H^s(\Sph^n)$ in the non-critical case $\rho\neq 0,1$. In the critical cases $\rho\in\{0,1\}$ we prove Rubin's Conjectures~4.4 and~4.7 on the failure of endpoint Sobolev regularity for the inverse transforms.
\end{abstract}

\section{Introduction}

Let $n \geq 2$ and let $\Sph^n \subset \R^{n+1}$ denote the unit sphere equipped
with the standard Riemannian structure. In his seminal paper on generalized
Minkowski--Funk transforms and small denominators on the sphere \cite{Rubin00},
Boris Rubin introduced a family of operators $M_t^\alpha$ on $\Sph^n$ which
interpolate between several classical transforms of integral geometry and
fundamental solution operators for hyperbolic PDEs. For background on the
classical Funk/Minkowski--Funk (spherical Radon) transform and related spherical
integral transforms, see, for instance, Helgason \cite{Helgason10} and Rubin's
monograph \cite{Rubin15}. See also Dann \cite{Dann10} for historical background
on the Minkowski--Funk transform. For fixed $\alpha \in \C$
and $t = \cos\theta \in (-1,1)$, the operator $M_t^\alpha$ is defined by
integrating a function $f$ over the spherical cap of geodesic radius $\theta$
centered at $x\in\Sph^n$ with a specific fractional weight depending on $\alpha$;
see \cite{Rubin00} for the precise definition and basic properties.

The operators $M_t^\alpha$ are $\mathrm{SO}(n+1)$-equivariant and hence diagonal
with respect to the standard decomposition of $L^2(\Sph^n)$ into spherical
harmonics. Writing
\[
  L^2(\Sph^n) \;=\; \bigoplus_{j=0}^\infty \scr{H}_j,
\]
where $\scr{H}_j$ is the finite-dimensional space of spherical harmonics of
degree $j$, we have
\[
  \Mt{\alpha} Y = \mt{\alpha}(j) \, Y, \qquad Y \in \scr{H}_j,
\]
for a scalar multiplier $\mt{\alpha}(j)$ depending on $t$, $\alpha$, $n$, and $j$.
Rubin obtained explicit formulas for $\mt{\alpha}(j)$ in terms of
hypergeometric and associated Legendre functions and, more importantly for our
purposes, a detailed asymptotic expansion as $j\to\infty$.

The inversion and Sobolev mapping properties of $\Mt{\alpha}$ with $t$ fixed
are governed by the behavior of $\mt{\alpha}(j)$ for large $j$. In particular,
one wishes to control the size of $j^{\rho+1} \mt{\alpha}(j)$, where
\[
  \rho := \alpha + \frac{n-1}{2},
\]
so as to obtain precise information about the order of smoothing and the
possibility of bounded inversion. Rubin's arguments show that for rational
values of $\beta = \theta/\pi$ one can often understand the zero set of
$\mt{\alpha}(j)$ explicitly. When $\beta$ is irrational, however, his
asymptotic formula leads to a small denominator problem involving delicate
Diophantine properties of $\beta$ and certain shifts $r_0,r_1$. In the
non-critical case $\rho\neq 0,1$ this is formulated as Problem~3.8 in
\cite[Section~3]{Rubin00}, while in the critical cases $\rho\in\{0,1\}$ it gives rise
to Conjectures~4.4 and~4.7 in \cite[Section~4]{Rubin00}.

We address Rubin's small-denominator problem for generalized Minkowski--Funk transforms.
In the critical cases $\rho\in\{0,1\}$ we prove Rubin's Conjectures~4.4 and~4.7 on the failure of
endpoint Sobolev regularity, thereby completing Rubin's analysis of the generalized
Minkowski--Funk transforms at the borderline exponents. In the non-critical regime $\rho\neq 0,1$
we give a metric resolution of Rubin's small-denominator problem: for Lebesgue-almost every
irrational $\beta$ the relevant small-divisor inequality has infinitely many solutions for every
threshold $c>0$. As a consequence, the inverse transform $(M_t^\alpha)^{-1}$ fails to be bounded
on the natural Sobolev scale for almost every radius $t=\cos(\beta\pi)$. A key ingredient is a
reformulation of Rubin's two-sine small-divisor condition as a one-dimensional inhomogeneous
approximation problem of moving-target type. This translation requires a careful analysis of the
interaction between the oscillatory terms; once it is established, the metric conclusions follow
from appropriate Khintchine-type inputs. The paper was written on the occasion of Boris Rubin's
80th birthday.

\paragraph{Metric Diophantine input and perturbed approximation.}
At the metric level, several mapping and boundedness questions in this paper
reduce to inhomogeneous Diophantine approximation and can be viewed as
moving-target problems for irrational rotations, in the spirit of Khintchine's
theorem and its inhomogeneous variants (see, e.g., \cite{Cassels57,Szuez58}).
This applies not only to the endpoint statements in the critical cases
$\rho\in\{0,1\}$ (Rubin's Conjectures~4.4 and~4.7), but also to our non-endpoint
result in Theorem~\ref{thm:boundness}, which likewise admits a reduction to
Khintchine-type inequalities with restricted denominators and moving targets.
Moving-target extensions have attracted recent attention. In higher dimensions,
one has moving-target inhomogeneous results in the
Duffin--Schaeffer/Khintchine setting due to Hauke--Ram\'{\i}rez
\cite{HaukeRamirez24}. In dimension one, however, the moving-target setting is
more delicate and typically requires additional hypotheses, such as extra
divergence or structural restrictions on the targets; see
Michaud--Ram\'{\i}rez \cite{MR}. Our proofs follow the approach of
Michaud--Ram\'{\i}rez \cite{MR}, with minor modifications to accommodate a
two-parameter family of targets with both numerator and denominator dependence,
as well as denominators restricted to an arithmetic progression; see
Section~\ref{sec:moving-khintchine}.

Related numerator-dependent perturbations have recently been studied by
Hussain--Ward \cite{HussainWard} (in arbitrary dimension), who consider
approximation by perturbed rationals of the form $(\vec a+\epsilon_{\vec a,q})/q$
with $\epsilon_{\vec a,q}$. Under a
quantitative smallness hypothesis on $|\epsilon_{\vec a,q}|$, they obtain
Jarn\'{\i}k-type zero--full laws for Hausdorff $f$--measure, and ask in
\cite[Question~2]{HussainWard} how far this hypothesis can be relaxed.
Our moving-target input addresses a different regime in dimension one: we allow an
additional denominator-dependent shift $\gamma_q$ (not necessarily tending to $0$),
and we assume smallness for the structured numerator-dependent component by writing
$\gamma_{a,q}=\gamma_q+\epsilon_{a,q}$ with $\epsilon_{a,q}$ $q$-periodic in $a$
and $|\epsilon_{a,q}|\le C_0\psi(q)$. This setting typically lies outside the
smallness hypothesis of \cite{HussainWard} when $\gamma_q\not=o(\psi(q))$, yet we
still obtain a full Lebesgue-measure conclusion, allowing also for
denominators restricted to an arithmetic progression, under the divergence hypotheses of
Theorem~\ref{thm.KinchineMovingTarget},
via the moving-target framework of Michaud--Ram\'{\i}rez \cite{MR}.

\subsection*{Main results}

We write $H^s(\Sph^n)$ for the usual Sobolev space on the sphere associated with
the Laplace--Beltrami operator. Following Rubin, we denote by $\tilde H^s(\Sph^n)$
the subspace obtained by removing the finite-dimensional kernel of $M_t^\alpha$
(see Section~\ref{sec:setup} for details).

In the non-critical regime $\rho\neq 0,1$, Rubin showed that the boundedness of
$(M_t^\alpha)^{-1}$ from $\tilde H^{s+\rho+1}(\Sph^n)$ to $H^s(\Sph^n)$ would
follow from a certain small divisor condition on the asymptotic multipliers,
and he formulated an explicit Diophantine problem in terms of a two-sine
expression $F(j)$; see Problem~\ref{prob:Rubin} below. Our first main result
(Theorem~\ref{thm:main_sd}) resolves this problem for Lebesgue-almost every
$\beta$: for any fixed $c>0$, the inequality $|F(j)|<c$ has infinitely many
integer solutions $j$. As a consequence we obtain that for almost every radius
$t=\cos(\beta\pi)$ the inverse $(M_t^\alpha)^{-1}$ is not bounded from
$\tilde H^{s+\rho+1}(\Sph^n)$ to $H^s(\Sph^n)$ (Theorem~\ref{thm:boundness}).

In the critical cases $\rho=0$ and $\rho=1$, Rubin identified the boundedness
of $(M_t^\alpha)^{-1}$ between Sobolev spaces with the finiteness of solutions
to certain Khintchine-type inequalities, and proved sharp metric results for
all exponents except the endpoints. He conjectured that the endpoint bounds
fail for almost every $\beta$ (Conjectures~4.4 and~4.7 in \cite{Rubin00}). Our
second main result (Theorem~\ref{thm:conjectures}) proves both conjectures, and
thus settles the borderline Sobolev mapping problem for Rubin's generalized
Minkowski--Funk transforms.

\subsection*{Organization of the paper}

In Section~\ref{sec:setup} we recall Rubin's generalized Minkowski--Funk transforms, 
summarize his asymptotic formula for the multipliers, and formulate the associated
small denominator problem. We then state our almost-everywhere small divisor
result (Theorem~\ref{thm:main_sd}) and its Sobolev consequence
(Theorem~\ref{thm:boundness}), and review Rubin's critical cases and conjectures.
Section~\ref{sec:proof_main_sd} contains the proof of Theorem~\ref{thm:main_sd},
based on a variant of Khintchine's theorem with moving targets. The necessary
metric Diophantine input, including a two-parameter moving-target version of
Khintchine's theorem, is developed in Section~\ref{sec:moving-khintchine}.

\subsection*{Acknowledgements}
R. Han and Y. Rahimi are partially supported by NSF DMS-2143369.  We are grateful to Felipe Ram\'{\i}rez for helpful discussions and for pointing out a short proof of Theorem~\ref{thm.KinchineMovingTarget2} in the unrestricted-denominator case $(y,z)=(1,0)$, based on Lemma~9 of Cassels, as well as an extension to the case $z=0$ (i.e.\ denominators $q\equiv 0 \pmod y$); see Remark~\ref{rem:Felipe}. We also thank him for drawing our attention to \cite{HussainWard} and Question~2 therein.

\section{Rubin's transforms, small denominators, and bounded inversion}\label{sec:setup}

We briefly recall the notation and results needed from Rubin's work. Throughout
we fix $n\geq 2$, $\alpha \in \C$ with $\alpha \neq -\tfrac{n}{2}, -\tfrac{n}{2}-1,\dots$,
and write
\[
  \rho = \alpha + \frac{n-1}{2}.
\]
We also fix a radius parameter
\[
  t = \cos\theta \in (-1,1), \qquad \theta = \beta\pi, \quad 0 < \beta < 1,\ \beta \neq \frac12.
\]

For each $j\in\N\cup\{0\}$, let $\scr{H}_j$ denote the space of spherical
harmonics of degree $j$ on $\Sph^n$, with orthogonal projection $P_j$. The
Laplace--Beltrami operator $\Lap$ on $\Sph^n$ satisfies
\[
  -\Lap Y = \lambda_j Y, \qquad Y\in\scr{H}_j,\quad
  \lambda_j = j(j+n-1),
\]
so that $\{\scr{H}_j\}_{j\ge0}$ is an eigenspace decomposition for $-\Lap$.
We write $H^s(\Sph^n)$, $s\in\R$, for the usual $L^2$–based Sobolev space on
$\Sph^n$, defined as the domain of $(1-\Lap)^{s/2}$ with the graph norm
(equivalently, via the above spectral decomposition).

Rubin's transforms $\Mt{\alpha}$ act by
\[
  \Mt{\alpha} f
  \;=\;
  \sum_{j=0}^\infty \mt{\alpha}(j) P_j f,
\]
where the multipliers $\mt{\alpha}(j)$ are given explicitly in terms of
associated Legendre functions and hypergeometric functions
(see \cite[Section~2]{Rubin00}).

Following Rubin, we will sometimes work on the reduced Sobolev space obtained
by removing the kernel of $\Mt{\alpha}$.
For fixed $\alpha$ and $t$, we define
\[
  \tilde H^s(\Sph^n)
  :=
  \overline{\{ f\in H^s(\Sph^n) : f \perp \ker \Mt{\alpha} \}}^{\,H^s(\Sph^n)}.
\]
Equivalently, $\tilde H^s(\Sph^n)$ is the closure in $H^s(\Sph^n)$ of the
orthogonal complement of $\ker \Mt{\alpha}$. On $\tilde H^s(\Sph^n)$ the
operator $\Mt{\alpha}$ is injective.

\subsection{The non-critical case $\rho\neq 0,1$}\label{sec:non_critical}

The detailed asymptotics of $\mt{\alpha}(j)$ for large $j$ depend on the
parameter $\theta = \beta\pi$ and, in particular, on $\beta = \theta/\pi$.
Rubin introduces auxiliary parameters
\[
  r_m := \frac{\rho - 1 - \beta n + \beta}{2} - m\Bigl(\beta + \frac12\Bigr),
  \qquad m = 0,1,2,\dots,
\]
and obtains a full asymptotic expansion of $\mt{\alpha}(j)$ as $j \to \infty$.
For our purposes we only need the first two terms of this expansion.

\begin{proposition}[Rubin's asymptotics, truncated]\label{prop:rubin-asymp}
Let $t = \cos\theta$, $\theta = \beta\pi$ with $0<\beta<1$, $\beta\neq 1/2$,
and let $\rho = \alpha + (n-1)/2 \notin \{0,1\}$. Then there exists a constant
$c_\beta \neq 0$ (depending on $\alpha$, $n$ and $\beta$ but not on $j$) and
real parameters
\[
  r_0=\frac{\rho-1-\beta n+\beta}{2}, \qquad r_1=r_0-(\beta+1/2),
\]
such that
\begin{equation}\label{eq:rubin-asymp}
  j^{\rho+1} \mt{\alpha}(j)
  \;=\;
  c_\beta (1+O(j^{-1})) \biggl(
    j \sin\bigl(\pi(j\beta - r_0)\bigr)
    + \rho(\rho-1) \sin\bigl(\pi(j\beta - r_1)\bigr)
    + O(j^{-1})
  \biggr)
\end{equation}
as $j\to\infty$.
\end{proposition}

\begin{remark}
The precise form of $c_\beta$ and of the $O(j^{-1})$ error term is not important
for the discussion below; what matters is that $c_\beta\neq 0$. A more precise expansion, with explicit
coefficients $c_{m,\beta}(j,\rho)$ and higher order terms, can be found in
\cite[Section~2]{Rubin00}.
\end{remark}

Rubin uses \eqref{eq:rubin-asymp} to study the invertibility of $\Mt{\alpha}$
on Sobolev spaces $H^s(\Sph^n)$ by relating the growth or decay of
$j^{\rho+1}\mt{\alpha}(j)$ to Diophantine properties of $\beta$ and the
parameters $r_0,r_1$.

In the regime $\rho\neq 0,1$, the main obstructions to bounded inversion of
$\Mt{\alpha}$ come from the possibility that the quantity
\[
  F(\beta,j)
  :=
  j \sin\bigl(\pi(j\beta - r_0)\bigr)
  + \rho(\rho-1) \sin\bigl(\pi(j\beta - r_1)\bigr)
\]
becomes arbitrarily small along some subsequence $j\to\infty$. Indeed,
\eqref{eq:rubin-asymp} shows that
\[
  j^{\rho+1}\mt{\alpha}(j)
  \;=\;
  c_\beta (1+O(j^{-1}))\cdot \bigl(F(\beta,j) + O(j^{-1})\bigr),
\]
so that small denominators for the multipliers are encoded in the behavior of
$F(\beta,j)$.

Using this asymptotic function, Rubin proved the following.

\begin{theorem}[Rubin {\cite[Theorem~3.7]{Rubin00}}]
Let $t=\cos(\beta\pi)$ where $\beta\in \R\setminus \Q$, and let
$\rho = \alpha + (n-1)/2 \neq 0,1$. Then $(M_t^{\alpha})^{-1}$ is not bounded
from $H^{s+\rho+\mu}(\Sph^n)$ into $H^{s}(\Sph^n)$ for any $\mu\in [0,1)$.
\end{theorem}

For $\mu=1$ and irrational $\beta$, as Rubin remarked in \cite[Section~3]{Rubin00}, the boundedness of $(M_t^{\alpha})^{-1}$ from $\tilde H^{s+\rho+1}(\Sph^n)$ into $H^s(\Sph^n)$ leads to the following small denominator problem.

\begin{problem}[Rubin's small denominator problem]\label{prob:Rubin}
Fix $n\geq 2$, $\alpha\in\C$ with $\rho = \alpha + (n-1)/2 \neq 0,1$, and
$t = \cos(\beta\pi)$ with $0<\beta<1$, $\beta\neq 1/2$. Let $r_0,r_1$ be the
parameters appearing in Proposition~\ref{prop:rubin-asymp}. Consider the
inequalities, for $j\in\N$,
\begin{equation}\label{eq:rubin-317}
  \bigl|F(\beta,j)\bigr|
  =
  \bigl| j \sin\bigl(\pi(j\beta - r_0)\bigr)
      + \rho(\rho-1) \sin\bigl(\pi(j\beta - r_1)\bigr)\bigr|
  < c,
\end{equation}
where $c>0$ is a sufficiently small constant.

\smallskip

Rubin asks for a description of those $\beta$ for which:
\begin{enumerate}
  \item There exists $c>0$ such that \eqref{eq:rubin-317} has only finitely many
    integer solutions $j$, so that $j^{\rho+1}\mt{\alpha}(j)$ stays uniformly
    away from zero for all sufficiently large $j$; and
  \item For every $c>0$, the inequality \eqref{eq:rubin-317} has infinitely many
    integer solutions $j$, forcing the occurrence of small denominators.
\end{enumerate}
\end{problem}

In this note, we solve this problem for almost every $\beta$.

\begin{theorem}\label{thm:main_sd}
 For Lebesgue-almost every $\beta\in(0,1)$, for any $c>0$, the inequality \eqref{eq:rubin-317} has infinitely many solutions in $j\in\N$.
\end{theorem}

We postpone the proof of this theorem to Section~\ref{sec:proof_main_sd}.
As a first application, we record the following consequence for the boundedness
of $\Mt{\alpha}$ on Sobolev spaces.

\begin{theorem}\label{thm:boundness}
Let $n\ge2$ and $\alpha\in\C$ with $\rho = \alpha + (n-1)/2\neq 0,1$, and let
$t=\cos(\beta\pi)$ with $\beta\in(0,1)$ irrational. For Lebesgue-almost every
such $\beta$, the operator $(M_t^{\alpha})^{-1}$ is not bounded from
$\tilde H^{s+\rho+1}(\Sph^n)$ into $H^{s}(\Sph^n)$.
\end{theorem}

\begin{remark}
Theorems~\ref{thm:main_sd} and~\ref{thm:boundness} together give a metric
answer to Problem~\ref{prob:Rubin} in the non-critical regime $\rho\neq 0,1$.
Indeed, for Lebesgue-almost every $\beta\in(0,1)$ the small divisor inequality
\eqref{eq:rubin-317} has infinitely many integer solutions for every $c>0$, so
that $j^{\rho+1}\mt{\alpha}(j)$ approaches $0$ along a subsequence and the
inverse transform $(M_t^\alpha)^{-1}$ fails to be bounded from
$\tilde H^{s+\rho+1}(\Sph^n)$ to $H^s(\Sph^n)$. 
\end{remark}

We now turn to the complementary, critical cases $\rho\in\{0,1\}$, where the
multipliers simplify and Rubin's conjectures concern the borderline Sobolev
mapping properties of $(M_t^\alpha)^{-1}$.

\subsection{The critical case $\rho\in \{0,1\}$ and Rubin's Conjectures 4.4 and 4.7}\label{sec:Rubin-conj}

In Section~\ref{sec:non_critical} we worked under the standing assumption
\[
  \rho = \alpha + \frac{n-1}{2} \notin \{0,1\}.
\]
Rubin treats the complementary, \emph{critical} cases
\[
  \alpha = \frac{1-n}{2} \quad (\rho = 0), \qquad
  \alpha = \frac{3-n}{2} \quad (\rho = 1),
\]
separately in \cite[Section~4]{Rubin00}. These values correspond, respectively, to the
solution operator for the spherical wave equation and to the spherical section
transform, and they lead to particularly simple multiplier formulas. In this
regime the small denominator problem can be expressed in terms of classical
Diophantine inequalities.

Write $t = \cos(\beta\pi)$ with $\beta \in (0,1)$ irrational and let
$H^s(\Sph^n)$ denote the usual Sobolev spaces. Rubin shows that for
$\alpha = (1-n)/2$ and $\alpha = (3-n)/2$ the boundedness of $(M_t^\alpha)^{-1}$
between Sobolev spaces is equivalent to the finiteness of solutions to certain
Diophantine inequalities of Khintchine type.

More precisely, for $\mu \geq 1$ Rubin introduces the inequalities
\begin{align}
  \label{eq:Rubin-48}
  \bigl\| q\beta + \tfrac12\bigr\| &< \frac{c}{q^\mu}, \\
  \label{eq:Rubin-49}
  \bigl\| (q-\tfrac12)\beta + \tfrac12\bigr\| &< \frac{c}{q^\mu}, \\
  \label{eq:Rubin-410}
  \| q\beta \| &< \frac{c}{q^\mu}, \\
  \label{eq:Rubin-411}
  \bigl\| (q-\tfrac12)\beta\bigr\| &< \frac{c}{q^\mu},
\end{align}
for $q\in\N$ and some constant $c>0$, where $\nnorm{\cdot}$ denotes the distance
to the nearest integer. He proves the following equivalence:
\begin{itemize}
    \item $(M_t^{(1-n)/2})^{-1}$ is bounded from $H^{s+\mu}(\Sph^n)$ to
    $H^{s}(\Sph^n)$ if and only if
\eqref{eq:Rubin-48} (for $n$ odd) or \eqref{eq:Rubin-49} (for $n$ even) has only
finitely many solutions $q\in\N$ for some $c>0$. 
    \item $(M_t^{(3-n)/2})^{-1}$ is bounded from $H^{s+1+\mu}(\Sph^n)$ to
    $H^{s}(\Sph^n)$ if and only if
\eqref{eq:Rubin-410} (for $n$ odd) or \eqref{eq:Rubin-411} (for $n$ even) has
only finitely many solutions $q\in\N$ for some $c>0$.
\end{itemize}

Using the metric theory of Diophantine approximation, he then proves the
following.

\begin{theorem}[Rubin {\cite[Corollaries~4.3, 4.5, 4.6]{Rubin00}}]
Let $t=\cos(\beta\pi)$ with $\beta\in (0,1)$ irrational.
\begin{itemize}
  \item For $\mu>1$, $(M_t^{(1-n)/2})^{-1}$ is bounded
    from $H^{s+\mu}(\Sph^n)$ to $H^{s}(\Sph^n)$ for almost every $\beta\in(0,1)$.
  \item For $\alpha = (3-n)/2$ and $\mu>1$, $(M_t^{(3-n)/2})^{-1}$ is bounded
    from $H^{s+1+\mu}(\Sph^n)$ to $H^s(\Sph^n)$ for almost every $\beta\in(0,1)$.
  \item For $\alpha = (3-n)/2$, $n$ odd, and $\mu=1$, one even has
    $(M_t^{(3-n)/2})^{-1} \in L(H^{s+2}(\Sph^n),H^s(\Sph^n))$ for almost no
    $\beta\in(0,1)$, i.e. the set of $\beta$ for which this boundedness holds
    has Lebesgue measure $0$.
\end{itemize}
\end{theorem}

\begin{rubinconjecture}[Rubin {\cite[Conjecture~4.4]{Rubin00}}]\label{conj:Rubin44}
For every $s\in\R$ and $n\geq 2$, the set of $\beta\in(0,1)$ such that
\[
  (M_t^{(1-n)/2})^{-1}
  \in L\bigl(H^{s+1}(\Sph^n), H^{s}(\Sph^n)\bigr)
\]
has Lebesgue measure zero. Equivalently, for Lebesgue-almost every
$\beta\in(0,1)$ the operator $(M_t^{(1-n)/2})^{-1}$ fails to be
bounded from $H^{s+1}(\Sph^n)$ to $H^{s}(\Sph^n)$.
\end{rubinconjecture}

\begin{rubinconjecture}[Rubin {\cite[Conjecture~4.7]{Rubin00}}]\label{conj:Rubin47}
For $n\geq 2$ even and $s\in\R$, the set of $\beta\in(0,1)$ such that
\[
  (M_t^{(3-n)/2})^{-1}
  \in L\bigl(H^{s+2}(\Sph^n), H^{s}(\Sph^n)\bigr)
\]
has Lebesgue measure zero. Equivalently, for almost every $\beta\in(0,1)$ the
operator $(M_t^{(3-n)/2})^{-1}$ fails to be bounded from
$H^{s+2}(\Sph^n)$ to $H^{s}(\Sph^n)$.
\end{rubinconjecture}

Our second main result is the following.

\begin{theorem}\label{thm:conjectures}
Conjectures~4.4 and 4.7 of Rubin in \cite{Rubin00} are true. 
\end{theorem}
\begin{proof}
   It suffices to show for $\mu=1$, for a.e. $\beta$, for any $c>0$, the inequalities \eqref{eq:Rubin-48}, \eqref{eq:Rubin-49} and \eqref{eq:Rubin-411} have infinitely many solutions.
    First note that this is true for \eqref{eq:Rubin-48} due to the inhomogeneous Diophantine inequality of \cite{Szuez58}. In fact, Szüsz proved that for any $\gamma\in \R$, and any non-increasing $\psi: \N\to (0,\infty)$ such that $\sum_q\psi(q)=\infty$, the inequality
    \[\|q\beta-\gamma\|<\psi(q)\] 
    has infinitely many solutions for a.e. $\beta$. Szüsz's result applies to \eqref{eq:Rubin-48} with $\gamma=1/2$, since clearly $\sum_q 1/q=\infty$.

    Regarding \eqref{eq:Rubin-49}: after the change of variable $\beta\to 2\beta$, it suffices to show for any $c>0$, the inequality
    \begin{align}\label{eq:rubin-49_odd}
    \bigl\|q\beta+\tfrac{1}{2}\bigr\|<\frac{c}{q}
    \end{align}
    has infinitely many odd solutions for a.e. $\beta$.  This is a fixed-target inhomogeneous Khintchine problem with restricted denominators. In our setting we can simply apply Theorem~\ref{thm.KinchineMovingTarget2} with $\gamma_{a,q} = 1/2$.
    The analysis of \eqref{eq:Rubin-411} is analogous to that of \eqref{eq:Rubin-49}.
\end{proof}

\section{Proof of Theorem \ref{thm:main_sd}}\label{sec:proof_main_sd}

        The proof uses crucially the following lemma:
\begin{lemma}\label{lem:key}
For any given $x$ and $B$, and $A=2\Z$ or $A=2\Z+1$, for any $c>0$, there exists a full measure set of $\beta$ such that there are infinitely many solutions with $k\in \Z$ and $j\in A$ to the following equation:
\begin{align}\label{eq:even}
    \left\|\frac{\beta}{4}-\frac{k+x}{j}+\frac{B}{j^2}\cos\frac{4\pi k}{j}\right\|\leq \frac{c}{j^2}.
\end{align}
\end{lemma}
Lemma~\ref{lem:key} may be viewed as a
variant of Khintchine’s moving–target problem. We postpone its proof to Sec. \ref{sec:moving-khintchine}. 
Here, we will show how it implies Theorem \ref{thm:main_sd}. Recall 
\begin{align}
    \rho=\alpha+\frac{n-1}{2},
\end{align}
and
\begin{align}
    r_0=\frac{\rho-1-\beta n+\beta}{2}, \text{ and } r_1=r_0-(\beta+\frac{1}{2}).
\end{align}

Assume $n$ is odd (or even), and apply Lemma \ref{lem:key} to $x=(\rho-1)/4$, $B=\rho(\rho-1)/\pi$, and $A=2\Z$ (or $2\Z+1$ respectively), for a.e. $\beta$, there are infinitely many $j$'s such that
\begin{align}\label{eq:sd_1}
    \left\|\frac{\beta}{4}-\frac{(\rho-1)/4+k}{2j+n-1}+\frac{\rho(\rho-1)}{\pi(2j+n-1)^2}\cos\frac{4\pi k}{2j+n-1}\right\|\leq \frac{c}{j^2}.
\end{align}
This implies that there exists an integer $m\in \Z$:
\begin{align}\label{eq:sd_1'}
    \left|\frac{\beta}{4}-\frac{(\rho-1)/4+k}{2j+n-1}+\frac{\rho(\rho-1)}{\pi(2j+n-1)^2}\cos\frac{4\pi k}{2j+n-1}-m\right|\leq \frac{c}{j^2},
\end{align}
or equivalently
\begin{align}\label{eq:sd_1''}
    \left|\beta-\frac{\rho-1+4k}{2j+n-1}+\frac{4\rho(\rho-1)}{\pi(2j+n-1)^2}\cos\frac{4\pi k}{2j+n-1}-4m\right|\leq \frac{4c}{j^2}.
\end{align}
Note that \eqref{eq:sd_1''} implies
\begin{align}
    \left|\frac{4k}{2j+n-1}-\left(\beta-\frac{\rho-1}{2j+n-1}\right)+4m\right|\leq \frac{4c+C_{\rho,n}}{j^2}.
\end{align}
Therefore,
\begin{align}\label{eq:cos=cos+ep2}
    \cos\frac{4\pi k}{2j+n-1}
    =&\cos \left(\beta\pi-\frac{\pi(\rho-1)}{2j+n-1}\right)+\varepsilon_1 \notag\\
    =&\cos(\beta\pi)\cos\left(\frac{\pi(\rho-1)}{2j+n-1}\right)+\sin(\beta\pi)\sin\left(\frac{\pi(\rho-1)}{2j+n-1}\right)+\varepsilon_1 \notag\\
    =&\cos(\beta\pi)+\varepsilon_2,
\end{align}
where 
\[|\varepsilon_1|\leq \frac{4\pi c+C_{\rho,n}}{j^2}, \text{ and } |\varepsilon_2|\leq \frac{C_{\rho,n}}{j}.\]
Plugging \eqref{eq:cos=cos+ep2} into \eqref{eq:sd_1''} yields
\begin{align}
    \left|\beta-\frac{\rho-1+4k}{2j+n-1}+\frac{4\rho(\rho-1)}{\pi(2j+n-1)^2}\cos(\pi\beta)+\frac{\varepsilon_3}{j^2}-4m\right|\leq \frac{4c}{j^2},
\end{align}
where 
\[|\varepsilon_3|\leq C_{\rho,n}\cdot \varepsilon_2\leq \frac{(C_{\rho,n})^2}{j} \leq c \quad \text{for large enough }  j.\]
This implies
\begin{align}
    \left|\beta-\frac{\rho-1+4k}{2j+n-1}+\frac{4\rho(\rho-1)}{\pi(2j+n-1)^2}\sin\pi(\beta+1/2)-4m\right|\leq \frac{4c+|\varepsilon_3|}{j^2} \leq \frac{5c}{j^2},
\end{align}
and hence
\begin{align}\label{eq:sd_2}
    \left|(2j+n-1)\beta-(\rho-1+4k)+\frac{4\rho(\rho-1)}{\pi(2j+n-1)}\sin(\pi(\beta+1/2))-4m(2j+n-1)\right|\leq \frac{10c}{j},
\end{align}
for $j$ large enough. 
Noting that
\begin{align}
    \left|\frac{4\rho(\rho-1)}{2j+n-1}-\frac{2\rho(\rho-1)}{ j}\right|\leq \frac{C_{\rho,n}}{j^2} \leq \frac{c}{j} \quad \text{for large enough }  j.
\end{align}
Then \eqref{eq:sd_2} implies
\begin{align}
    \left|(2j+n-1)\beta-(\rho-1+4k)+\frac{2\rho(\rho-1)}{ \pi j}\sin(\pi(\beta+1/2))-4m(2j+n-1)\right|\leq \frac{12c}{j}.
\end{align}
Dividing by $2$ and re-arranging the terms we get
\begin{align}
    \left|j\beta+\frac{n-1}{2}\beta-\left(\frac{\rho-1}{2}+2k\right)+\frac{\rho(\rho-1)}{\pi j}\sin(\pi(\beta+1/2))-2m(2j+n-1)\right|\leq \frac{6c}{j},
\end{align}
Recall that $\frac{(n-1)\beta-(\rho-1)}{2}=-r_0$, hence the inequality above turns into:
\begin{align}\label{eq:sd_3}
    \left|j\beta-r_0-2k+\frac{\rho(\rho-1)}{\pi j}\sin(\pi(\beta+1/2))-2m(2j+n-1)\right|\leq \frac{6c}{j}.
\end{align}
This implies
\begin{align}\label{eq:jbeta-r0}
    |j\beta-r_0-2k-2m(2j+n-1)|\leq \frac{6c+C_{\rho}}{j}\ll 1. 
\end{align}
Hence 
\begin{align}\label{eq:sd_4}
    \sin\pi(j\beta-r_0)
    =&\sin\pi(j\beta-r_0-2k-2m(2j+n-1))\\
    =&\pi(j\beta-r_0-2k-2m(2j+n-1))+\frac{C_1}{j^3},
\end{align}
where $|C_1|\leq (6c+ C_{\rho})^3$.  Thus \eqref{eq:sd_3} together with \eqref{eq:sd_4} imply
\begin{align}\label{eq:sd_5}
    \left|\sin\pi(j\beta-r_0)+\frac{\rho(\rho-1)}{j}\sin(\pi(\beta+1/2))\right|\leq \frac{7\pi c}{j}.
\end{align}
Recall $\beta+\frac{1}{2}=r_0-r_1=(j\beta-r_1)-(j\beta-r_0)$.
Hence
\begin{align}\label{eq:sd_6}
    \frac{\rho(\rho-1)}{j}\sin(\pi(\beta+1/2))
    =&\frac{\rho(\rho-1)}{j}\left(\sin\pi((j\beta-r_1)-(j\beta-r_0))\right)\notag\\
    =&\frac{\rho(\rho-1)}{j}[\sin \pi(j\beta-r_1)\cos\pi(j\beta-r_0)-\cos\pi(j\beta_1-r_1)\sin \pi(j\beta-r_0)]\notag\\
    =&\frac{\rho(\rho-1)}{j}\sin\pi(j\beta-r_1)+\frac{C_2}{j^2},
\end{align}
with $|C_2|\leq 2(6c+C_{\rho})$. Above, we used \eqref{eq:jbeta-r0} and that
\begin{align}
    \cos\pi(j\beta-r_0)=1-O_{c,\rho}(j^{-2}), \text{ and } \sin\pi(j\beta-r_0)=O_{c,\rho}(j^{-1}).
\end{align}
Plugging \eqref{eq:sd_6} into \eqref{eq:sd_5}, we have
\begin{align}
    \left|\sin\pi(j\beta-r_0)+\frac{\rho(\rho-1)}{j}\sin\pi(j\beta-r_1)\right|\leq \frac{8\pi c}{j},
\end{align}
which is 
\[\left|j\sin\pi(j\beta-r_0)+\rho(\rho-1)\sin\pi(j\beta-r_1)\right|\leq 8\pi c,\]
as claimed. \qed

\section{Khintchine's theorem for two-parameter moving targets}\label{sec:moving-khintchine}

Note that Lemma~\ref{lem:key} (see Section~\ref{sec:proof_main_sd}) may be viewed as a
variant of Khintchine’s moving–target problem. 
Before formulating the version relevant for our argument, we recall a result of 
Michaud--Ramírez.  
For a sequence of real numbers $\gamma = (\gamma_q)$ and an approximation function 
$\psi : \mathbb{N} \to \mathbb{R}_{>0}$, define
\[
    W(\psi,\gamma)
    =
    \Bigl\{
        \beta \in [0,1] :
        \nnorm{ \beta - \tfrac{a+\gamma_q}{q} }
        < \tfrac{\psi(q)}{q}
        \text{ for infinitely many } (a,q)\in  \Z\times \mathbb{N}
    \Bigr\},
\]
where $\nnorm{\cdot}$ denotes distance to the nearest integer, and let $\mathrm{m}$ denote
Lebesgue measure on $\T=\R/\Z$.

\begin{theorem}[Moving–Target Khintchine]\cite[Theorem 1]{MR}
\label{thm.KinchineMovingTarget}
Let $\gamma=(\gamma_q)$ be any sequence of real numbers, and let $\psi$ be monotone decreasing.  
Suppose that there exists $\varepsilon>0$ and an integer $k\ge2$ such that
\[
    \sum_{q=1}^{\infty}
    \frac{\psi(q)}
    {\sqrt{\log q \, (\log\log q)\, (\log\log\log q)\cdots
    (\log^{(k)} q)^{\,1+\varepsilon}}}
    = \infty,
\]
where $\log^{(k)}$ denotes the $k$-fold iterated logarithm.  
Then $\mathrm{m}(W(\psi,\gamma)) = 1$.  
In particular, the weaker divergence condition 
\[
    \sum_{q=1}^{\infty} \psi(q)\, (\log q)^{-1/2-\varepsilon} = \infty
\]
already suffices to guarantee full measure.
\end{theorem}

The setting of Lemma~\ref{lem:key} is close in spirit to
Theorem~\ref{thm.KinchineMovingTarget}, but it does not fit that framework
directly. In the Michaud--Ram\'{\i}rez theorem, the target shift $\gamma_q$
depends only on the denominator $q$. In our application, by contrast, the relevant perturbation varies with both the
numerator and the denominator: we are led to a family of moving targets of the form
\[
    \bigl(\gamma_{a,q}\bigr)_{a\in\mathbb{Z},\, q\in\mathbb{N}},
    \qquad
    \gamma_{a,q}=\gamma_q+\epsilon_{a,q},
\]
where the error terms $\epsilon_{a,q}$ are uniformly small (e.g.\ $|\epsilon_{a,q}|\ll \psi(q)$)
and, for each fixed $q$, are $q$-periodic in $a$, i.e.\ $\epsilon_{a+q,q}=\epsilon_{a,q}$.

Accordingly, we record below a two-parameter and congruence-restricted version
(denominators in a fixed arithmetic progression) of the moving-target statement
that suffices for our purposes. Rather than aiming for maximal generality, the
result is tailored to the structure arising here and may be viewed as an
adaptation of the approach of Michaud--Ram\'{\i}rez to this slightly richer
family of targets.

\begin{theorem}
\label{thm.KinchineMovingTarget2}
Let $\psi$ be monotone decreasing and satisfy the divergence hypotheses of
Theorem~\ref{thm.KinchineMovingTarget}.
Let $(\gamma_q)_{q\in\mathbb{N}}$ be any sequence of real numbers. For each $q$,
let $(\epsilon_{a,q})_{a\in\mathbb{Z}}$ be a $q$-periodic family (i.e.\ $\epsilon_{a+q,q}=\epsilon_{a,q}$)
satisfying $|\epsilon_{a,q}|\le C_0\psi(q)$ for some constant $C_0>0$, and set
$\gamma_{a,q}:=\gamma_q+\epsilon_{a,q}$.
Then, for any $y\in \N$ and any $z\in\{0,1,\dots,y-1\}$,
\[
    W_{y,z}(\psi,\gamma)
    =
    \Bigl\{
        \beta\in \T:
        \nnorm{\beta - \tfrac{a+\gamma_{a,q}}{q}}
        < \tfrac{\psi(q)}{q}
        \text{ for infinitely many }(a,q)\in\Z\times\N \text{ with } q\equiv z\!\!\!\pmod y
    \Bigr\}
\]
has full Lebesgue measure, i.e.\ $\mathrm{m}(W_{y,z}(\psi,\gamma)) = 1$.
\end{theorem}

Lemma~\ref{lem:key} is obtained as a direct corollary of 
Theorem~\ref{thm.KinchineMovingTarget2} by taking
\[
    \psi(q) = \frac{c}{q},
    \qquad 
    \gamma_{a,q}
    = x - \frac{B}{q}\cos \Bigl(\frac{4\pi a}{q}\Bigr),
\]
and $(y,z)\in \{(2,0), (2,1)\}$,
which gives precisely the moving targets required for our application.
Note that a variant of a fixed target Khinchine inequality suffices for our application to Lemma~\ref{lem:key}, since our $\gamma_q\equiv x$ is a constant.

\begin{remark}[A short proof in the case $z=0$, suggested by Ram\'{\i}rez]\label{rem:Felipe}
Assume first $(y,z)=(1,0)$, and assume $\psi(q)\to 0$.
By \cite[Lemma~9]{Cassels57}, for any sequence of intervals whose lengths tend to $0$,
the Lebesgue measure of the associated limsup set is unchanged if all interval lengths are multiplied
by a fixed constant. In particular,
\[
\mathrm{m}\bigl(W_{1,0}(\psi,\gamma)\bigr)
=
\mathrm{m}\bigl(W_{1,0}((C_0+1)\psi,\gamma)\bigr).
\]
Moreover, if
\[
\nnorm{\beta-\tfrac{a+\gamma_q}{q}}<\tfrac{\psi(q)}{q},
\]
then using $\gamma_{a,q}=\gamma_q+\epsilon_{a,q}$ and $|\epsilon_{a,q}|\le C_0\psi(q)$ we have
\[
\nnorm{\beta-\tfrac{a+\gamma_{a,q}}{q}}
\le
\nnorm{\beta-\tfrac{a+\gamma_q}{q}}+\tfrac{|\epsilon_{a,q}|}{q}
<
\tfrac{(C_0+1)\psi(q)}{q}.
\]
Hence
\[
W(\psi,(\gamma_q))\subseteq W_{1,0}((C_0+1)\psi,\gamma).
\]
Since $\psi$ satisfies the divergence hypotheses of Theorem~\ref{thm.KinchineMovingTarget} and is decreasing,
the moving-target result of Michaud--Ram\'{\i}rez implies $\mathrm{m}(W(\psi,(\gamma_q)))=1$, and therefore
$\mathrm{m}(W_{1,0}((C_0+1)\psi,\gamma))=1$. The Cassels scaling lemma then yields
$\mathrm{m}(W_{1,0}(\psi,\gamma))=1$.

\medskip
Ram\'{\i}rez also noted a short argument for handling the restriction $q\equiv 0\pmod y$.
In the one-parameter case (targets depending only on $q$), one has
\[
W_{y,0}(\psi,(\gamma_q)) \;=\; T_y^{-1}\Bigl(W(\psi_y,(\gamma_{yq}))\Bigr),
\qquad T_y(\beta)=y\beta\pmod 1,\quad \psi_y(q):=\psi(yq),
\]
so $\mathrm{m}(W_{y,0}(\psi,(\gamma_q)))=\mathrm{m}(W(\psi_y,(\gamma_{yq})))$ since $T_y^{-1}$ is
measure-preserving. As $\psi$ is monotone and satisfies the divergence hypothesis of
Theorem~\ref{thm.KinchineMovingTarget}, the same holds for $\psi_y$, and hence
Michaud--Ram\'{\i}rez gives $\mathrm{m}(W_{y,0}(\psi,(\gamma_q)))=1$.
Finally, the perturbation bound $|\epsilon_{a,q}|\le C_0\psi(q)$ implies
$W_{y,0}(\psi,(\gamma_q))\subseteq W_{y,0}((C_0+1)\psi,\gamma)$, and Cassels' scaling lemma
\cite[Lemma~9]{Cassels57} (applied to the subsequence $q\equiv 0\pmod y$) yields
$\mathrm{m}(W_{y,0}(\psi,\gamma))=1$.

We do not pursue an analogue of this shortcut for $z\neq 0$ here.
\end{remark}

\begin{proof}[Proof of Theorem \ref{thm.KinchineMovingTarget2}]
We apply the criterion of Michaud--Ram\'{\i}rez (Proposition~\ref{prop:all01} below) to the
sets appearing in the limsup description of $W_{y,z}(\psi,\gamma)$, and then use the
local-to-global lemma of Beresnevich--Dickinson--Velani (Proposition~\ref{prop:BDV}) to upgrade
a quantitative lower bound on \linebreak $\mathrm{m}(W_{y,z}(\psi,\gamma)\cap U)$ for every nonempty open
set $U$ to the full-measure conclusion. Concretely, we verify the hypotheses of
Proposition~\ref{prop:all01} by establishing: a size estimate for $E_q$, a pairwise intersection
bound for $E_q\cap E_r$, a uniform density estimate for $E_q\cap U$, and finally the divergence
condition \eqref{eq:divergence}.

We may assume that $\psi(q)\to 0$ as $q\to\infty$; otherwise $\psi(q)\ge c>0$ for all large $q$,
in which case $E_q=\T$ for all large $q$ and $\mathrm{m}(W_{y,z}(\psi,\gamma))=1$ is immediate.

\medskip
\noindent\textbf{Setup.}
Write
\[
W_{y,z}(\psi,\gamma)=\limsup_{q\to\infty}E_{qy+z}
=\bigcap_{Q=1}^{\infty}\;\bigcup_{q\ge Q} E_{qy+z},
\]
where
\[
E_q=\bigcup_{a=0}^{q-1} I_{q,a},\qquad
I_{q,a}:=\Bigl\{\beta \in\T:\ \nnorm{\beta- \tfrac{a+\gamma_{a,q}}{q}}
        < \tfrac{\psi(q)}{q} \Bigr\}.
\]
Recall that $\gamma_{a,q}=\gamma_q+\epsilon_{a,q}$ with $|\epsilon_{a,q}|\le C_0\psi(q)$ and
$\epsilon_{a+q,q}=\epsilon_{a,q}$ for all $a\in\Z$, $q\in\N$. In particular, for each fixed $q$ the
quantity $\gamma_{a,q}$ depends only on $a\bmod q$, so in the definition of $W_{y,z}(\psi,\gamma)$
it suffices to restrict to $0\le a\le q-1$, which yields the above representation in terms of $E_q$.

\medskip
\noindent\textbf{Size of the basic sets.}

\begin{lemma}[Size of the basic sets]\label{lem:Eq-size}
For all sufficiently large $q$,
\[
\mathrm{m}(E_q)=2\psi(q).
\]
\end{lemma}

\begin{proof}
The upper bound is immediate since each $\mathrm{m}(I_{q,a})=\frac{2\psi(q)}{q}$.
For the lower bound it suffices to show that the intervals $I_{q,a}$ are pairwise disjoint
for $q$ large.

Suppose $I_{q,a}\cap I_{q,a'}\neq \emptyset$ with $a\not\equiv a'\pmod q$. Then
\[
\left\|\frac{a+\gamma_q+\epsilon_{a,q}}{q}-\frac{a'+\gamma_q+\epsilon_{a',q}}{q}\right\|
\leq \frac{2\psi(q)}{q}.
\]
By the triangle inequality,
\[
\left\|\frac{a-a'}{q}\right\|
\leq \frac{2\psi(q)+|\epsilon_{a,q}|+|\epsilon_{a',q}|}{q}
\leq (2C_0+2)\frac{\psi(q)}{q}<\frac{1}{q},
\]
for $q$ large enough (since $\psi(q)\to 0$). This contradicts $a\not\equiv a' \pmod q$.
Hence the $I_{q,a}$ are disjoint for large $q$, and $\mathrm{m}(E_q)=q\cdot \frac{2\psi(q)}{q}=2\psi(q)$.
\end{proof}

\medskip
\noindent\textbf{Pairwise overlap bounds.}

\begin{lemma}[Overlap estimates]  \label{lem:EqEr-overlap}
For all sufficiently large integers $q$ and all $1 \le r < q$, we have
\[
\mathrm{m}(E_q \cap E_r)\leq 4(C_0+1)\left(\mathrm{m}(E_q)\, \mathrm{m}(E_r)+\frac{\gcd(q,r)}{q}\mathrm{m}(E_q)\right).
\]
\end{lemma}

\begin{proof}
We use the notation of \cite{MR}. Let
\[
\Delta:=2\max\Bigl(\frac{\psi(q)}{q}, \frac{\psi(r)}{r}\Bigr), 
\qquad 
\delta:=2\min\Bigl(\frac{\psi(q)}{q},\frac{\psi(r)}{r}\Bigr).
\]

For $q,r$ sufficiently large, Lemma~\ref{lem:Eq-size} implies the families $\{I_{q,a}\}_a$
and $\{I_{r,b}\}_b$ are pairwise disjoint. Thus if $\beta\in E_q\cap E_r$, then there exist unique
$0\le a\le q-1$ and $0\le b\le r-1$ such that $\beta\in I_{q,a}\cap I_{r,b}$.
Moreover,
\[
\mathrm{m}(I_{q,a}\cap I_{r,b})\le \delta.
\]
Hence it suffices to bound the number of pairs $(a,b)$ such that $I_{q,a}\cap I_{r,b}\neq\varnothing$.

If $I_{q,a} \cap I_{r,b} \neq \varnothing$, then the centres satisfy
\begin{align} \label{eq:solution01}
\left\|
\frac{a+\gamma_q+\epsilon_{a,q}}{q} 
- \frac{b + \gamma_r+\epsilon_{b,r}}{r}
\right\|
\le
\frac{\psi(q)}{q} + \frac{\psi(r)}{r}.
\end{align}
By the triangle inequality this implies
\begin{align} \label{eq:solution01'}
\left\|
\frac{a+\gamma_q}{q} 
- \frac{b + \gamma_r}{r}
\right\|
\le\;
&\frac{\psi(q)+|\epsilon_{a,q}|}{q} + \frac{\psi(r)+|\epsilon_{b,r}|}{r}\notag\\
\le\; 
&(C_0+1)\left(\frac{\psi(q)}{q} + \frac{\psi(r)}{r}\right)\le (C_0+1)\Delta.
\end{align}

Since $\frac{a}{q}-\frac{b}{r}\in(-1,1)$ for $0\le a\le q-1$ and $0\le b\le r-1$,
there exists an integer $k$ (depending only on $q,r$ up to at most two choices) such that
\[
\left|
\frac{a+\gamma_q}{q} 
- \frac{b + \gamma_r}{r}
- k \right| \le (C_0+1)\Delta,
\]
equivalently,
\begin{align}\label{eq:ar-bq_1}
|(ar-bq)-kqr+(r\gamma_q-q\gamma_r)|<(C_0+1)qr\Delta.
\end{align}
For each such $k$, since $\gcd(q,r)\mid (ar-bq)$, there are at most
\[
2(C_0+1)\frac{qr\Delta}{\gcd(q,r)}+1
\]
admissible values of $(ar-bq)$. For each admissible value, there are at most $\gcd(q,r)$ solutions $(a,b)$.
Therefore, the total number of pairs $(a,b)$ with $I_{q,a}\cap I_{r,b}\neq\varnothing$ is at most
\[
4(C_0+1)\left(qr\Delta+\gcd(q,r)\right).
\]
Hence
\begin{align}
\mathrm{m}(E_q\cap E_r)
&\le 4(C_0+1)\left(qr\Delta+\gcd(q,r)\right)\delta\\
&\le 4(C_0+1)\left(4\psi(q)\psi(r)+\frac{\gcd(q,r)\cdot 2\psi(q)}{q}\right)\\
&\le 4(C_0+1)\left(\mathrm{m}(E_q)\, \mathrm{m}(E_r)+\frac{\gcd(q,r)}{q}\mathrm{m}(E_q)\right),
\end{align}
using Lemma~\ref{lem:Eq-size} in the last line.
\end{proof}

\medskip
\noindent\textbf{Density in open sets.}

\begin{lemma}[Open Set Overlap] \label{lem:EqU-overlap}
For every nonempty open set $U \subset \T$, there exists $q_0 := q_0(U) > 0$ such that
\[
\mathrm{m}(E_q \cap U) \;\ge\; \tfrac{1}{2}\,\mathrm{m}(E_q)\,\mathrm{m}(U)
\quad\text{for all } q \ge q_0.
\]
\end{lemma}

\begin{proof}
First assume that $U=I$ is an open interval.
Let $\Gamma_{q}:=(\Z+\gamma_q)/q$ mod $1$ and set $N:=\#(\Gamma_q\cap I)$.
Enumerate the points of $\Gamma_q\cap I$ in increasing order as
\begin{align}\label{eq:order_Gamma_q_cap_I}
0\leq \inf_{x\in I} x< \frac{a_1+\gamma_q}{q}-w_1<\cdots <\frac{a_N+\gamma_q}{q}-w_N<\sup_{x\in I} x<1,
\end{align}
with integers $w_j$. These points are $1/q$-spaced. Since $|\epsilon_{a,q}|\le C_0\psi(q)$ and $\psi(q)\to 0$,
the ordering remains unchanged after the perturbation for $q$ large enough:
\[
\frac{a_1+\gamma_q+\epsilon_{a_1,q}}{q}-w_1<\cdots <\frac{ a_N+\gamma_q+\epsilon_{a_N,q}}{q}-w_N.
\]
Moreover, for $q$ large enough we have
\begin{align}
\frac{a_2+\gamma_q+\epsilon_{a_2,q}-\psi(q)}{q}-w_2
&\ge \frac{a_2+\gamma_q-(C_0+1)\psi(q)}{q}-w_2\\
&\ge \frac{a_1+\gamma_q}{q}-w_1+\frac{1-(C_0+1)\psi(q)}{q}
\ge \frac{a_1+\gamma_q}{q}-w_1,
\end{align}
and similarly
\[
\frac{a_{N-1}+\gamma_q+\epsilon_{a_{N-1},q}+\psi(q)}{q}-w_{N-1}
<\frac{a_N+\gamma_q}{q}-w_N.
\]
Therefore, for $2\le j\le N-1$,
\[
\left(\frac{a_j+\gamma_q+\epsilon_{a_j,q}-\psi(q)}{q}-w_j,\,
      \frac{a_j+\gamma_q+\epsilon_{a_j,q}+\psi(q)}{q}-w_j\right)\subset I,
\]
and hence
\begin{align}\label{eq:Eq_cap_I}
\mathrm{m}(E_q\cap I)\ge (N-2)\cdot \frac{2\psi(q)}{q}.
\end{align}
From \eqref{eq:order_Gamma_q_cap_I} we also have
\begin{align}\label{eq:mI}
\mathrm{m}(I)\le (N+1)\cdot \frac{1}{q}.
\end{align}
Combining \eqref{eq:Eq_cap_I}, \eqref{eq:mI}, and Lemma~\ref{lem:Eq-size} gives, for $q$ large enough,
\[
\mathrm{m}(E_q\cap I)\ge \frac{N-2}{N+1}\,\mathrm{m}(E_q)\,\mathrm{m}(I)\ge \frac{3}{4}\,\mathrm{m}(E_q)\,\mathrm{m}(I),
\]
which implies the desired inequality (with $\frac12$) for intervals. The general case follows by approximating
an open set $U$ from inside by a finite disjoint union of open intervals.
\end{proof}

\medskip
\noindent\textbf{Divergence.}
To complete the verification of the hypotheses of Proposition~\ref{prop:all01} (stated below),
it remains to check the divergence condition \eqref{eq:divergence}. In our setting, the
parameter $\eta(q)$ arising from the overlap bound \eqref{eq:pairwise} can be bounded above by a
quantity of divisor type (ultimately by $d(qy+z)$), and hence \eqref{eq:divergence} reduces to
the divergence of a divisor-weighted series. The next two lemmas supply precisely the required
divergence estimate.

\begin{lemma}\label{lem:sum+lowerbound'}
Let $y\ge 1$ be fixed and let $z$ be an integer. Let $d(n)$ denote the number of
divisors of $n$, and let $f:[1,\infty)\to[1,\infty)$ be an increasing function.
Then, for all $x$ sufficiently large, one has
\[
\sum_{\substack{n\le x\\ n\equiv z\!\!\!\pmod y}}
\frac{1}{f(\log d(n))\,d(n)}
\ \gg_y\
\frac{x}{f((\log 2)\log\log x)\,\sqrt{\log x}} .
\]
\end{lemma}

\begin{proof}[Proof of Lemma \ref{lem:sum+lowerbound'}]
Define
\[
A_{y,z}(x)
:= \Bigl\{1\le n\le x:\ n\equiv z \!\!\!\pmod y,\ 
\log_2 d(n)\le \log\log x\Bigr\},
\]
where $\log_2$ denotes the logarithm to base $2$.  Since $f$ is increasing and
$\log d(n) = (\log 2)\,\log_2 d(n)$, we have $\log d(n)\le (\log 2)\log\log x$
for all $n\in A_{y,z}(x)$, and hence
\begin{align}\label{eq:Ay-lower}
\sum_{\substack{n\le x\\ n\equiv z\!\!\!\pmod y}}
\frac{1}{f(\log d(n))\,d(n)}
&\ge
\sum_{n\in A_{y,z}(x)}
\frac{1}{f(\log d(n))\,d(n)} \notag\\
&\ge
\frac{1}{f((\log 2)\log\log x)}
\sum_{n\in A_{y,z}(x)} \frac{1}{d(n)} .
\end{align}

Next we compare the $1/d(n)$--mass of $A_{y,z}(x)$ with that of the whole
progression. If $n\notin A_{y,z}(x)$, then $\log_2 d(n)>\log\log x$, hence
$d(n)>(\log x)^{\log 2}$ and therefore
\[
\frac{1}{d(n)} \le (\log x)^{-\log 2}.
\]
Consequently,
\begin{align}\label{eq:complement}
\sum_{\substack{n\le x\\ n\equiv z\!\!\!\pmod y\\ n\notin A_{y,z}(x)}} \frac{1}{d(n)}
\ \le\
\#\Bigl\{n\le x:\ n\equiv z\!\!\!\pmod y\Bigr\}(\log x)^{-\log 2}
\ \ll_y\ \frac{x}{(\log x)^{\log 2}}.
\end{align}

On the other hand, we have the following uniform lower bound for the full sum:
\begin{equation}\label{eq:1overd-lower}
\sum_{\substack{n\le x\\ n\equiv z\!\!\!\pmod y}} \frac{1}{d(n)}
\ \gg_y\ \frac{x}{\sqrt{\log x}}.
\end{equation}
Indeed, letting $g=(y,z)$ and writing $y=g y_1$, $z=g z_1$ with $(y_1,z_1)=1$, the
congruence $n\equiv z\pmod y$ is equivalent to $n=g m$ with $m\equiv z_1\pmod{y_1}$,
so that
\[
\sum_{\substack{n\le x\\ n\equiv z\!\!\!\pmod y}} \frac{1}{d(n)}
=
\sum_{\substack{m\le x/g\\ m\equiv z_1\!\!\!\pmod{y_1}}} \frac{1}{d(gm)}.
\]
Since $g\mid y$, the divisor function is submultiplicative and gives
$d(gm)\le d(g)\,d(m)$.
Therefore
\[
\sum_{\substack{n\le x\\ n\equiv z\!\!\!\pmod y}} \frac{1}{d(n)}
=
\sum_{\substack{m\le x/g\\ m\equiv z_1\!\!\!\pmod{y_1}}} \frac{1}{d(gm)}
\ \ge\
\frac{1}{d(g)}\sum_{\substack{m\le x/g\\ m\equiv z_1\!\!\!\pmod{y_1}}} \frac{1}{d(m)}.
\]
For fixed $y_1$ and $(y_1,z_1)=1$, Selberg--Delange theory in arithmetic
progressions applied to $1/d(m)$ yields
\[
\sum_{\substack{m\le X\\ m\equiv z_1\!\!\!\pmod{y_1}}} \frac{1}{d(m)}
\ \asymp_{y_1}\ \frac{X}{\sqrt{\log X}},
\]
so taking $X=x/g$ gives, for $x$ sufficiently large,
\[
\sum_{\substack{n\le x\\ n\equiv z\!\!\!\pmod y}} \frac{1}{d(n)}
\ \gg_{y_1}\ \frac{1}{d(g)}\cdot \frac{x/g}{\sqrt{\log(x/g)}}
\ =\ \frac{x}{g\,d(g)\,\sqrt{\log(x/g)}}.
\]
Since $y$ is fixed and $y_1$ ranges over divisors of $y$, we may absorb the
dependence on $y_1$ into the implied constant and write
\[
\sum_{\substack{n\le x\\ n\equiv z\!\!\!\pmod y}} \frac{1}{d(n)}
\ \gg_{y}\ \frac{x}{g\,d(g)\,\sqrt{\log(x/g)}}.
\]
Finally, using $g\le y$, $d(g)\le d(y)$, and $\sqrt{\log(x/g)}\asymp_y \sqrt{\log x}$
(for $x$ sufficiently large, depending on $y$), we obtain
\[
\sum_{\substack{n\le x\\ n\equiv z\!\!\!\pmod y}} \frac{1}{d(n)}
\ \gg_{y}\ \frac{x}{\sqrt{\log x}},
\]
which is \eqref{eq:1overd-lower}.

Combining \eqref{eq:complement} with \eqref{eq:1overd-lower} and using $\log 2>1/2$, we obtain that for all $x$ sufficiently large,
\[
\sum_{n\in A_{y,z}(x)}\frac{1}{d(n)}
\ \gg_y\ \frac{x}{\sqrt{\log x}}.
\]
Substituting this into \eqref{eq:Ay-lower} completes the proof.
\end{proof}

\medskip
\noindent
We now choose a specific increasing function $F$ arising from Theorem~\ref{thm.KinchineMovingTarget} and use Lemma~\ref{lem:sum+lowerbound'} to transfer divergence from a logarithmic weight to a divisor weight.

Fix $\varepsilon>0$ and an integer $k\ge 2$, and define
\[
F(x)\ :=\ x(\log x)(\log\log x)\cdots \bigl(\log^{(k-2)}x\bigr)^{1+\varepsilon},
\]
where $\log^{(m)}$ denotes the $m$-fold iterated logarithm. 

\begin{lemma}\label{lem:sum_psi_dn_infty}
For fixed $y\ge 1$ and integer $z$, one has
\[
\sum_{\substack{n\ge 1\\ n\equiv z\!\!\!\pmod y}}
\frac{\psi(n)}{F(\log d(n))\,d(n)}=\infty.
\]
\end{lemma}

\begin{proof}[Proof of Lemma \ref{lem:sum_psi_dn_infty}]
As in \cite[Section~4]{MR}, one has
\begin{equation}\label{eq:18}
\sum_{\substack{n\le x\\ n\equiv z\!\!\!\pmod y}}
\frac{1}{F((\log 2)\log\log n)\sqrt{\log n}}
\ \asymp_y\
\frac{x}{F(\log\log x)\sqrt{\log x}}.
\end{equation}

Define the two nonnegative sequences
\[
g(n):=\frac{1}{F(\log d(n))\,d(n)},
\qquad
h(n):=\frac{1}{F((\log 2)\log\log n)\sqrt{\log n}},
\]
and their partial sums along the progression $n\equiv z\pmod y$,
\[
G(x):=\sum_{\substack{n\le x\\ n\equiv z\!\!\!\pmod y}} g(n),
\qquad
H(x):=\sum_{\substack{n\le x\\ n\equiv z\!\!\!\pmod y}} h(n).
\]
By Lemma~\ref{lem:sum+lowerbound'} with $f=F$, and using
$F((\log 2)\log\log x)\asymp F(\log\log x)$ for large $x$, we obtain
\[
G(x)\ \gg_y\ \frac{x}{F(\log\log x)\sqrt{\log x}}\ \asymp_y\ H(x)
\qquad (x\ \text{sufficiently large}).
\]
In particular, there exist $x_0$ and $c=c(y)>0$ such that
\begin{equation}\label{eq:GgecH}
G(x)\ \ge\ c\,H(x)\qquad\text{for all }x\ge x_0.
\end{equation}

Since $\psi$ is decreasing, Abel summation (applied to the sequence
$a_n:=g(n)\mathbf{1}_{n\equiv z\!\!\!\pmod y}$) gives, for every integer $N\ge 1$,
\begin{equation}\label{eq:Abel-g}
\sum_{\substack{n\le N\\ n\equiv z\!\!\!\pmod y}} \psi(n)\,g(n)
=
\psi(N)\,G(N)+\sum_{m< N}\bigl(\psi(m)-\psi(m+1)\bigr)\,G(m).
\end{equation}
Applying the same identity to $b_n:=h(n)\mathbf{1}_{n\equiv z\!\!\!\pmod y}$ yields
\begin{equation}\label{eq:Abel-h}
\sum_{\substack{n\le N\\ n\equiv z\!\!\!\pmod y}} \psi(n)\,h(n)
=
\psi(N)\,H(N)+\sum_{m< N}\bigl(\psi(m)-\psi(m+1)\bigr)\,H(m).
\end{equation}
Using \eqref{eq:GgecH} for all $m\ge x_0$ and the nonnegativity of
$\psi(m)-\psi(m+1)$, we obtain for all $N$ sufficiently large
\[
\sum_{\substack{n\le N\\ n\equiv z\!\!\!\pmod y}} \psi(n)\,g(n)
\ \ge\
c\sum_{\substack{n\le N\\ n\equiv z\!\!\!\pmod y}} \psi(n)\,h(n)
\;-\;O_{y,\psi}(1).
\]
In particular, there exist constants $c=c(y)>0$ and $C=C(y,z,\psi,F)>0$ such that for all
sufficiently large $N$,
\begin{equation}\label{eq:sumd>sumlogn}
\sum_{\substack{n\le N\\ n\equiv z\!\!\!\pmod y}}
\frac{\psi(n)}{F(\log d(n))\,d(n)}
\ \ge\ c
\sum_{\substack{n\le N\\ n\equiv z\!\!\!\pmod y}}
\frac{\psi(n)}{F((\log 2)\log\log n)\sqrt{\log n}}
\ -\ C.
\end{equation}

Since, by assumption,
\begin{equation}\label{eq:psi-full}
\sum_{q=1}^{\infty}\frac{\psi(q)}{F(\log\log q)\sqrt{\log q}}=\infty,
\end{equation}
and since $\psi$ is decreasing while $F(\log\log q)\sqrt{\log q}$ is increasing for $q$ sufficiently large,
the summand in \eqref{eq:psi-full} is eventually decreasing; hence restricting to any fixed residue class modulo $y$ preserves divergence. In particular,
\begin{equation}\label{eq:psi-assumption}
\sum_{\substack{q=1\\ q\equiv z\!\!\!\pmod y}}^{\infty}
\frac{\psi(q)}{F(\log\log q)\sqrt{\log q}}=\infty.
\end{equation}
Using again $F((\log 2)\log\log q)\asymp F(\log\log q)$, the right-hand side in \eqref{eq:sumd>sumlogn} has partial sums tending to $+\infty$ along $N\to\infty$, hence so does the left-hand side. This proves the claimed divergence.
\end{proof}

\medskip
\noindent
The next proposition (from \cite{MR}) converts the size/overlap estimates above together with the divergence in Lemma~\ref{lem:sum_psi_dn_infty} into a quantitative lower bound for $\mathrm{m}(W_{y,z}(\psi,\gamma)\cap U)$ for every open set $U$.

\begin{proposition}\cite[Proposition 8]{MR} \label{prop:all01}
For each $q \ge 1$, let $A_q \subset X$, where $X$ is a metric space with a finite measure $\mu$ such that every open set is $\mu$-measurable. Suppose $\eta : \mathbb{N}^2 \to [1,\infty)$ is such that
\begin{equation}\label{eq:pairwise}
\mu(A_q \cap A_r) \le \kappa\, \mu(A_q)\big(\mu(A_r) + \eta(q,r)\big)
\qquad (1 \le r \le q)
\end{equation}
for some $\kappa > 0$. Suppose $U \subset X$ is an open set and that
\begin{equation}\label{eq:densityU}
\mu(A_q \cap U) \ge \frac12 \mu(A_q)\mu(U)
\end{equation}
for all sufficiently large $q$. Let
\[
\eta(q) = \sum_{r \le q} \eta(q,r),
\]
and let $h : [1,\infty) \to [1,\infty)$ be increasing and such that
\begin{equation}\label{eq:h-condition}
h(2q) \ll h(q)
\qquad\text{and}\qquad
\sum_{\ell \ge 0} \frac{1}{h(2^\ell)} < \infty.
\end{equation}
If
\begin{equation}\label{eq:divergence}
\sum_{q=1}^\infty \frac{\mu(A_q)}{\eta(q)\, h(\eta(q))} = \infty,
\end{equation}
then
\[
\mu\big( \limsup_{q \to \infty} A_q \cap U \big) \ge \frac{\mu(U)^2}{4\kappa}.
\]
\end{proposition}

\begin{proposition}[Beresnevich--Dickinson--Velani, {\cite[ Lemma~6]{BDV}}]\label{prop:BDV}
Let $(X,d)$ be a metric space with a finite measure $\mu$ such that every open set is $\mu$-measurable. Let $A$ be a Borel subset of $X$ and let $f : \mathbb{R}_{+} \to \mathbb{R}_{+}$ be an increasing function with $f(x) \to 0$ as $x \to 0$. If for every open set $U \subset X$ we have
\[
\mu(A \cap U) \ge f\big(\mu(U)\big),
\]
then
\[
\mu(A) = \mu(X).
\]
\end{proposition}

Continuing the proof of Theorem \ref{thm.KinchineMovingTarget2}, put
\[
h(x)\ :=\ F(\log x).
\]
Clearly $h$ satisfies the conditions in \eqref{eq:h-condition}.

Now define, for each $q\ge 1$,
\[
A_q :=E_{qy+z}.
\]
Then $\limsup_{q\to\infty}A_q = W_{y,z}(\psi,\gamma)$.

Moreover, Lemma~\ref{lem:EqU-overlap} implies that for every nonempty open set $U\subset \T$,
\[
\mathrm{m}(A_q\cap U)\ \ge\ \tfrac12\,\mathrm{m}(A_q)\,\mathrm{m}(U)
\qquad\text{for all sufficiently large }q,
\]
so \eqref{eq:densityU} holds for the sequence $(A_q)$.

By Lemma~\ref{lem:EqEr-overlap}, for all sufficiently large $q$ and all $1\le r<q$ we have
\[
    \mathrm{m}(A_q\cap A_r)\leq 4(C_0+1)\left(\mathrm{m}(A_q)\, \mathrm{m}(A_r)+\frac{\gcd(qy+z,ry+z)}{qy+z}\mathrm{m}(A_q)\right).
\]
Hence the inequality \eqref{eq:pairwise} holds with $\kappa=4(C_0+1)$ and
\[
\eta(q,r):=\frac{\gcd(qy+z,ry+z)}{qy+z}.
\]

Therefore,
\[
\eta(q)=\sum_{1\le r\le q}\eta(q,r)
\leq \frac{1}{qy+z}\sum_{1\le m\le qy+z}\gcd(qy+z,m)
=\sum_{e\mid (qy+z)}\frac{\varphi(e)}{e}
\ \le\  d(qy+z).
\]
Since $h$ is increasing, we have
\[
h(\eta(q))\le h(d(qy+z))=F(\log d(qy+z)),
\]
and hence
\[
\eta(q)\,h(\eta(q))\le d(qy+z)\,F(\log d(qy+z)).
\]
Also, by Lemma~\ref{lem:Eq-size}, for all sufficiently large $q$,
\[
\mathrm{m}(A_q)=\mathrm{m}(E_{qy+z})=2\psi(qy+z).
\]
Consequently, for all sufficiently large $q$,
\[
\frac{\mathrm{m}(A_q)}{\eta(q)\,h(\eta(q))}
\ \ge\
\frac{2\psi(qy+z)}{d(qy+z)\,F(\log d(qy+z))}.
\]
By Lemma~\ref{lem:sum_psi_dn_infty}, it follows that
\[
\sum_{q=1}^\infty \frac{\mathrm{m}(A_q)}{\eta(q)\,h(\eta(q))}=\infty.
\]
Therefore Proposition~\ref{prop:all01} implies that for every open set $U\subset \T$,
\[
\mathrm{m}\bigl(W_{y,z}(\psi,\gamma)\cap U\bigr)\ \ge\ \frac{\mathrm{m}(U)^2}{16(C_0+1)}.
\]
Since $U$ was arbitrary, Proposition~\ref{prop:BDV} yields
\[
\mathrm{m}\bigl(W_{y,z}(\psi,\gamma)\bigr)=1,
\]
completing the proof.
\end{proof}

\end{document}